\documentclass[11pt]{article}
\usepackage{amsfonts}
\usepackage{amsfonts}  
\usepackage{amsthm}
\usepackage{amsmath}
\usepackage{color}

\title{Dirichlet problems for fully nonlinear  equations with "subquadratic" Hamiltonians} 

\author{Isabeau Birindelli \\
Dipartimento di Matematica, Sapienza Universit\`a\  di Roma
\and
  Fran\c{c}oise Demengel\\
  D\'epartement de Math\'ematiques,
Universit\'e\  de Cergy-Pontoise
  \and
   Fabiana  Leoni\\ 
   Dipartimento di Matematica, Sapienza Universit\`a\  di Roma}

\date{}

\catcode`@=11 \@addtoreset{equation}{section} \catcode`@=12

\newtheorem{theo}{Theorem}[section]
\newtheorem{prop}[theo]{Proposition}
\newtheorem{rema}[theo]{Remark}
\newtheorem{defi}[theo]{Definition}

\newtheorem{lemme}[theo]{Lemma}

\def\R{\mathbb  R}
\def\grad{\nabla}

\begin{document}
\maketitle
\begin{abstract}
For a class of  fully nonlinear equations having second order operators which may be singular or degenerate when the gradient of the solutions vanishes, and  having first order terms with power growth, we prove  the existence and uniqueness of suitably defined viscosity solution of Dirichlet problem and we further show that it is a Lipschitz continuous function.
\smallskip

\end{abstract}

\section{Contents of the paper}
 In this paper we prove some existence,  uniqueness and Lipschitz regularity  results for  solutions of the  Dirichlet problems associated with a class of fully nonlinear equations, whose principal part is given by second order operators which are singular or degenerate elliptic at the points where the gradient of the solution vanishes. Precisely, let us consider the problem
 \begin{equation}\label{eq1}
\left\{ \begin{array}{c}
-  F( \nabla u,  D^2 u)+ b(x) | \nabla u |^\beta  + \lambda |u|^\alpha u = f \quad \hbox{ in } \ \Omega\\[2ex]
 u = \varphi  \quad \hbox{ on } \ \partial \Omega 
 \end{array}\right.
 \end{equation}
where $\Omega\subset \R^N$ is  a bounded, ${\mathcal C}^2$ open set  and the operator $F$ satisfies the structural assumptions
\begin{itemize}
 \item[(H1)]
 $F:\R^N\setminus \{0\}\times {\mathcal S}_N\to \R$  is a continuous function, ${\mathcal S}_N$ being the set of $N\times N$ symmetric matrices;

\item[(H2)] $F(p,M)$ is homogeneous of degree $\alpha > -1$ wrt $p$, positively homogeneous of degree $1$ wrt $M$,   and satisfies, for some constants $A\geq a>0$,  
\begin{equation}\label{eqF1}
a |p|^\alpha {\rm tr}(N) \leq F(p, M+N)-F(p, M) \leq  A |p|^\alpha {\rm tr}(N)
\end{equation} 
 for any $M, N\in {\mathcal S}_N$, with $N\geq 0$, and, for some $c> 0$,
  \begin{equation}\label{eqF2}|F(p, M)- F(q, M) | \leq c  |M| \, ||p|^\alpha-|q|^\alpha| 
  \end{equation}
   for any $p, q \in \R^N\setminus \{0\}$, and  $M \in {\mathcal S}_N$. 
 \end{itemize}
We further assume that the first order coefficient  $b$ is   Lipschitz continuous,  the forcing term $f$ is bounded and continuous, and the boundary datum $\varphi$ is Lipschitz continuous. Moreover, we will consider the cases $\lambda>0$ and $\beta\in (0, \alpha+2]$. In the case $\alpha=0$, the conditions on $\beta$ reduce to $0<\beta\leq 2$, hence the terminology "subquadratic".
 
   At least in the case $\alpha <0$,   the definition of viscosity solution we adopt must be clarified. We use the definition  firstly  introduced  in \cite{BD1},  which is equivalent to the usual one in the case $\alpha \geq 0$, and allows not to test points where the gradient of the test function is zero, except in the locally constant case, when $\alpha<0$.
    \begin{defi} Let $\Omega $ be an open set in $\R^N$, let $f: \Omega \times \R\to \R$ be a continuous function. An 
  upper (lower) semicontinuous function $u$ a is a subsolution (supersolution) of 
  $$-  F( \nabla u,  D^2 u)+ b(x) | \nabla u |^\beta  =  f (x, u)\quad \hbox{ in } \Omega
  $$
 if for any $x_0\in \Omega$, either $u$ is locally constant around $x_0$ and $f(x_o, u(x_0))\geq 0 (\leq) $, or for any test function $\phi$ of class ${\mathcal C}^2$ around $x_0$ such that $u-\phi$ has a local maximum (minimum) point in $x_0$ and $\nabla \phi (x_0) \neq 0$, then 
   $$ -  F( \nabla \phi (x_0),  D^2 \phi(x_0))+ b(x_0) | \nabla \phi (x_0)|^\beta  \leq  (\geq)f(x_0,u(x_o))\, .$$
A solution is a continuous function which is both a  supersolution and  a subsolution. 
     \end{defi}

 The results of the present paper have been announced in \cite{BDL}, where we consider solutions $u=u_\lambda$ of problem \eqref{eq1} with either $\varphi=0$ or $\varphi=+\infty$ and the behavior of $u_\lambda$ as $\lambda\to 0$ is studied. Clearly, the first step to perform this analysis is a detailed description of the existence, uniqueness and regularity properties of the solutions of problems \eqref{eq1} in the case $\lambda>0$, which is precisely the object of this paper. Our results can be summarized in the following main theorem.
 \begin{theo}\label{main}
 Under the above assumptions, problem \eqref{eq1} has a unique viscosity solution, which is Lipschitz continuous up to the boundary.
 \end{theo}
 
 We recall that the restrictions on the exponent $\beta$ for having existence of solutions already appear in the non singular nor degenerate case $\alpha=0$, when generalized solutions satisfying the boundary conditions in the viscosity sense can be constructed, but loss of boundary conditions may occur, see \cite{BDa, CDLP}.
 
 Theorem \ref{main} is obtained as a consequence of the classical Perron's method, which in turn relies on  the 
 comparison principle given by Theorem \ref{comp}. Observe that when $\alpha <0$ and the gradient of the involved 
 test functions  is $0$, the information on the solution is recovered through the result in Lemma \ref{lemalphaneg}, 
 which is analogous to one used in \cite{BD1} in the "sublinear"  case i.e. $\beta \leq \alpha+1$. Moreover when $b$ 
is not constant it is necessary to check that the gradient of the test functions be uniformly bounded. This comes from 
a priori interior Lipschitz estimates, proved in Theorem \ref{theolip},  which are  of independent interest.  
   
 We prove Lipschitz estimates up to the boundary for solutions of the Dirichlet problem \eqref{eq1}. Our proof follows by the Ishii--Lions technique, see \cite{IL}, which has to be adapted to the present singular/degenerate case. We borrow ideas from \cite{BCI, BDcocv}, and we first prove H\"older estimates and then push the argument  up to the Lipschitz result. Regularity results for degenerate elliptic equations have been obtained also in \cite{IS}, where equations having only a principal term of the form $F(\nabla u, D^2u)$ have been considered. However, we observe that any scaling argument relying on the homogeneity of the operator  is  no applicable in our case when $\beta\neq \alpha+1$, due to the different homogeneities of the terms in the equation.  We also emphasize that  we make no assumptions on the sign of the first  order coefficient $b$, meaning that the estimates we obtain rely only on the ellipticity properties of the second order term, despite its singularity or  degeneracy. The estimates we obtain here are radically different from the local Lipschitz regularity result proved in \cite{BDL}, where only positive hamiltonians with "superlinear" exponent $\beta>\alpha +1$ are considered. In that case, the obtained Lipschitz estimates, which do not depend on the $L^\infty$ norm of the solution, are consequence of the coercivity of the first order term and require that the  forcing term $f$ is Lipschitz continuous.
 
 The Lipschitz estimates are proved in Section  \ref{lip}. In Section \ref{Scomp}, after the construction of sub and supersolutions vanishing on the boundary, we prove the comparison principle and obtain the proof of Theorem \ref{main} in the case $\varphi=0$. A Strong Maximum Principle and Hopf Principle are also included. Finally, the changes   to be done in order to prove Theorem \ref{main} for any boundary datum $\varphi$ are detailed in Section \ref{Nonh}.

\section{Lipschitz estimates.}\label{lip}
The main result in this section is the following Lipschitz type estimate, where we denote by $B_1$ the unit ball in $\R^N$.
     \begin{theo}\label{theolip} Suppose that  $u$ is a  bounded viscosity subsolution of 
      $$- F(\grad u,D^2 u)  + b(x)|\nabla u|^\beta \leq  g\quad \hbox{ in } B_1$$
and $v$ is a  bounded viscosity supersolution of 
$$- F(\nabla v, D^2 v)  + b(x)|\nabla v|^\beta \geq  f\quad \hbox{ in } B_1\, ,$$         
with $f$ and $g$ bounded and  $b$ Lipschitz continuous. 
Then,  for all $r<1$, there exists $c_r$ such that for all $(x,y)\in B_r^2$
        $$u(x)-v(y) \leq \sup_{B_1} (u-v)+ c_r |x-y|.$$
        \end{theo}
In order to prove Theorem \ref{theolip} we first obtain the following H\"older estimate:
       
\begin{lemme}\label{lem1} Under the hypothesis of Theorem \ref{theolip},  for any    $\gamma \in (0,1)$,   there exists
$c_{r, \gamma}>0$ such that for  all $(x,y)\in B_r^2$
\begin{equation}\label{holdest}
u(x)-v(y) \leq \sup_{B_1} (u-v)+ c_{r, \gamma}|x-y|^\gamma .
        \end{equation} 
\end{lemme}        
\begin{proof}[Proof of Lemma \ref{lem1}]
We borrow ideas from \cite{IL},  \cite{BCI}, \cite{BDcocv}. 
Fix $x_o \in B_r$, and define 
     $$\phi(x, y) = u(x) - v(y) -\sup_{B_1} ( u-v) -M|x-y|^\gamma -L (|x-x_o|^2+ |y-x_o|^2)$$
with $L= {4(|u|_\infty+ |v|_\infty)\over (1-r)^2}$  and $M= {(|u|_\infty+ |v|_\infty)+1\over \delta^\gamma}$,  
$\delta$ will be chosen later small enough depending only on the data and on universal constants.  
We want to prove that $\phi (x,y) \leq 0$ in $B_1$ which will imply the result, 
taking first $x= x_o$ and making $x_o$ vary. 
   
We argue by contradiction and suppose that $\sup_{B_1}  \phi(x, y)>0$. 
By the previous assumptions on $M$ and $L$  the supremum is achieved on $(\bar x, \bar y)$  which  belongs to  
$B_{1+r\over 2} ^2$ and  it is such that $0<|\bar x-\bar y| \leq \delta$. 
   
By Ishii's Lemma \cite{I1},  for all $\epsilon >0$ there   exist  $X_\epsilon $  and $Y_\epsilon $ in $S$  such that 
$ (q^x, X_\epsilon ) \in J^{2,+} u(\bar x), (q^y, -Y_\epsilon ) \in J^{2,-}v(\bar y)$
with 
           $$ q^x = \gamma M |\bar x-\bar y|^{\gamma-2} (\bar x-\bar y) +2 L (\bar x-x_o)$$
           $$q^y = \gamma M |\bar x-\bar y|^{\gamma-2} (\bar x-\bar y) -2L (\bar y-x_o).$$
Hence
\begin{equation}\label{equaz}- F(q^x,X_\epsilon)  + b(\bar x)|q^x|^\beta \leq  g(\bar x),\quad           - F(q^y,-Y_\epsilon)  + b(\bar y)|q^y|^\beta \geq  f(\bar y)
\end{equation}
Furthermore,            
\begin{eqnarray}\label{eqXYB}
             -({1\over \epsilon } +|B|) \left( \begin{array}{cc}
             I&0\\0& I\end{array} \right) &\leq& 
            \left( \begin{array}{cc} 
             X_\epsilon&0\\
             0& Y_\epsilon\end{array}\right) 
              -2L \left( \begin{array}{cc} 
             I&0\\
             0& I\end{array}\right)\\
 &&\leq\quad\quad \left( \begin{array}{cc} 
             B +2 \epsilon B^2&-B-2 \epsilon B^2\\
             -B-2 \epsilon B^2& B+2 \epsilon B^2\end{array}\right)\nonumber   \end{eqnarray}
with $B = M\gamma |\bar x- \bar y |^{\gamma-2} \left( I-(2-\gamma) {(\bar x -\bar y) \otimes ( \bar x-\bar y)\over |\bar x -\bar y|^2}\right)$. 
 It is immediate to see that,  as soon as $\delta$ is small enough, with our choice of $M$ and $L$, 
 there exists $c_1$ and $c_2$, such that
$$c_1\delta^{-\gamma}  | \bar x-\bar y|^{ \gamma-1}\leq |q^x|, |q^y | \leq c_2\delta^{-\gamma}  | \bar x-\bar y|^{ \gamma-1}.$$  
We take $\epsilon = {1 \over 8 |B| + 1} = {1\over 8M \gamma |\bar x-\bar y |^{\gamma-2}  (N-\gamma)+1}$ and drop 
the index $\epsilon $ for $X_\epsilon$ and $Y_\epsilon$.  By standard considerations on the eigenvalues of $B$, in particular note that $B+ \epsilon B^2$ has a negative eigenvalue less than  
$-{3\gamma(1-\gamma)\over 4} M|\bar x-\bar y |^{\gamma-2}$, and, using \eqref{eqXYB}, the following holds
$$ X+ Y \leq (2L+\epsilon ) I\ \mbox{ and }\ \inf\lambda_i(X+Y)\leq 4 \inf \lambda_i ( B + \epsilon B^2) \leq -3\gamma(1-\gamma)M|\bar x-\bar y |^{\gamma-2}.$$
Hence, as 
soon as $\delta$ is small enough,   for some constant $c$ depending only on $a$, $A$, $\gamma$ and $N$,   
\begin{equation}\label{181}
\begin{array}{lll}
              F(q^x,X)-F  (  q^x, -Y) &\leq&|q^x|^\alpha\left(  A  \sum _{\lambda_i >0} \lambda_i (X+Y) + a \sum_{ \lambda_i <0} \lambda_i (X+Y)\right)\\[2ex]
              &\leq& |q^x |^\alpha \left( 2AN (2L+1)
               - 3a M \gamma (1-\gamma)\right) |\bar x-\bar y |^{\gamma-2} \\[2ex]
               &\leq & -c \delta^{-\gamma(\alpha+1)} | \bar x -\bar y |^{\gamma(\alpha+1) -(\alpha+2)}.
               \end{array}
               \end{equation}
And the following standard inequalities hold: 
$$ \mbox{If } \ 0\leq \alpha <1,\  ||q^y|^\alpha   -|q^x|^\alpha | \leq |q^x-q^y|^\alpha\leq  cL^\alpha$$
$$ \mbox{If } \alpha \geq 1, \  | |q^y|^\alpha -|q^x|^\alpha | \leq |\alpha|  |q^x-q^y| ( |q^x| + |q^y| )^{\alpha-1}\leq c(\delta^{-\gamma}|\bar x-\bar y|^{(\gamma-1)})^{(\alpha-1)}$$
\begin{equation*}\begin{split} \mbox{If }  -1< \alpha<0 ,\  | |q^y|^\alpha -|q^x|^\alpha | \leq &\  |\alpha|  |q^x-q^y| \inf (|q^x|,  |q^y| )^{\alpha-1} \\
\leq& \ c(\delta^{-\gamma}|\bar x-\bar y|^{(\gamma-1)})^{(\alpha-1)}\end{split}
               \end{equation*}
This implies that,  for any $\alpha>-1$,
\begin{equation}\label{182}
| F(q^x, X) -F(q^y, X) | \leq c\max(\delta^{-\gamma} |\bar x-\bar y |^{\gamma-2}, \delta^{-\alpha\gamma}|\bar x-\bar y |^{\alpha\gamma-\alpha-1}).
\end{equation}
Next, we  need to evaluate $|b(\bar x) |q^x|^\beta -b(\bar y)|q^y|^\beta |$.  One easily has,   for some constant which depends only on $\beta $ and universal constants, 
\begin{eqnarray} \label{183}
|b(\bar x) |q^x|^\beta -b(\bar y)|q^y|^\beta |  &\leq& {\rm Lip } \ b\  c\delta^{-\gamma\beta} |\bar x-\bar y |^{\gamma\beta-\beta+1} \\
\nonumber&&+\ c |b |_\infty\max(1, \delta^{ -\gamma(\beta-1)} |\bar x-\bar y |^{ \gamma(\beta-1)-\beta+1}).
\end{eqnarray}
 Observe that choosing $\delta$ small, the terms in \eqref{182} and in \eqref{183} are of lower order with respect to
\eqref{181}. 
We then have, for  some constant $c$,  subtracting the inequalities in   \eqref{equaz},
              \begin{eqnarray*}
              -g(\bar x)  &\leq &  F(q^x,X) -b(\bar x) |q^x|^\beta  \\
              &\leq &    F(q^y,Y)   - c M^{1+ \alpha}  |\bar x-\bar y|^{\gamma-2+ ( \gamma-1) \alpha}- b ( \bar y) |q^y|^\beta  \\
              &\leq &-f(\bar y) - c \delta^{-\gamma(1+ \alpha)}  |\bar x-\bar y|^{\gamma(\alpha+1)-(2+\alpha)}
              \end{eqnarray*}
This is a contradiction with the fact that $f$ and $g$ are bounded, as soon as $\delta$ is small enough.  
\end{proof}
 \medskip                 
               
We are now in a position to prove   the  Lipschitz estimate. The proof proceeds analogously to the H\"older estimate
above, but with a modification in the term depending on  $|x-y|$ in the function $\phi(x,y)$.
\begin{proof}[Proof of  Theorem \ref{theolip}]                     
For fixed $\tau\in (0,\frac{1}{2})$ for $\alpha\leq 0$ and $\tau\in (0,{\inf (1, \alpha) \over 2}))$ for $\alpha>0$, and $s_o=(1+ \tau)^{\frac{1}{\tau}}$ we define
for $s\in (0,s_o)$ with \begin{equation}
 \label{omega}
 \omega(s)  = s-{s^{1+\tau}\over 2(1+ \tau)},
 \end{equation} 
which we extend continuously after $s_o$ by a constant. 
 
Note that $\omega(s) $ is  ${\cal C}^2$ on $s>0$,  $s< s_o$ and satisfies 
 $\omega^\prime >0$, $\omega^{\prime \prime} <0$  on $]0,1[$, and 
 $s>\omega (s) \geq {s\over 2}$.

As before in the H\"older case, with $L = \frac{ 4(|u|_\infty + |v|_\infty)+1}{(1-r)^2}$ and $M = \frac{(|u|_\infty+ |v|_\infty) +1}{\delta}$  we define  
$$\phi(x,y)= u(x) - v(y) -\sup_{B_1} (u-v) -M\omega(|x-y|) -L (|x-x_o|^2+ |y-x_o|^2).$$  
 Classically, as before, we suppose that there exists a maximum point $(\bar x, \bar y)$ such that $\phi(\bar x,\bar y)>0$, then by the assumptions on $M$, and $L$,  $\bar x, \bar y$ belong to $B( x_o, {1-r \over 2})$, hence they are interior points. 
This implies, using \eqref{holdest} in Lemma \ref{lem1} with $\gamma <1$ such that  $\frac{\gamma}{2} >  {\tau\over \inf (1,\alpha)}$ for $\alpha>0$ and $\frac{\gamma}{2} > \tau$ when $\alpha\leq 0$ that,  for some constant  $c_r$,

\begin{equation}\label{hh} L|\bar x-x_o|^2 \leq c_r  |\bar x-\bar y |^{\gamma}.\end{equation}
and then  one has 
$|\bar x-x_o | \leq \left({c_r  \over L}\right)^{1\over 2}  |\bar x-\bar y |^{\gamma\over 2}$.

Furthermore, for all $\epsilon >0$, 
 there   exist  $X_\epsilon $  and $Y_\epsilon $ in $S$  such that 
          $ (q^x, X_\epsilon ) \in J^{2,+} u(\bar x)$, $(q^y, -Y_\epsilon ) \in J^{2,-}v(\bar y)$ with
 $$ q^x= M \omega^\prime (|\bar x-\bar y|) {\bar x-\bar y\over |\bar x-\bar y|}  + L (\bar x-x_o), \ q^y= M \omega^\prime (|\bar x-\bar y|) {\bar x-\bar y\over |\bar x-\bar y|}  - L (\bar y-x_o).$$
 While $X_\epsilon$ and $Y_\epsilon$ satisfy \eqref{eqXYB} with 
 $$ B =M\left( {\omega^\prime (|\bar x-\bar y|) \over |\bar x-\bar y|} ( I-{ \bar x-\bar y \otimes \bar x-\bar y \over |\bar x-\bar y |^2}) + \omega^{\prime \prime } (|\bar x-\bar y|) { \bar x-\bar y \otimes \bar x-\bar y \over |\bar x-\bar y |^2}\right).$$
Note that as soon as $\delta$ is small enough 
 $ {M \over 2} \leq |q^x|, |q^y |\leq {3M \over 2} $. Also,  
 $M\omega^{\prime \prime} ( |\bar x-\bar y|) = -M {\tau\over 2} |\bar x-\bar y |^{\tau-1}$ is an eigenvalue of
 $B$ which is large negative as soon as $\delta$ is small enough. 

Taking $\epsilon = {1\over 8 |B|+1}$ and dropping the $\epsilon$ in the notations
arguing as in  the above proof, we get that there exists some constant $c$ such that 
\begin{equation}\label{vpneg}
         F(q^x, X)- F(q^x ,  -Y) \leq -c \delta^{-(1+ \alpha)} |\bar x-\bar y |^{\tau-1} .
\end{equation}
Note that by  (\ref{eqXYB}) using the explicit value of $B$ one has 
           \begin{equation}\label{majX}|X| + |Y| \leq c (M |\bar x-\bar y |^{-1} + 4(L+1) N) \leq c \delta^{-1}|\bar x -\bar y |^{-1}
           \end{equation}
as soon $\delta$ is small enough. 
            
              
On the other hand, using the mean value theorem, that  for some universal constant 
$$\mbox{ if }\ \alpha \geq 1,\ \mbox{ or }\ \alpha <0, \quad | |q ^x|^\alpha  -|q^y|^\alpha|  \leq  c  \delta^{-\alpha +1} |\bar x-\bar y |^{\gamma\over 2},$$  
while 
$$ \mbox{ if }\ 0< \alpha \leq 1,\  | |q ^x|^\alpha  -|q^y|^\alpha | \leq  c |\bar x-\bar y |^{\gamma\alpha \over 2}.$$ 
In each of these cases  one easily obtains, using ( \ref{majX})
               
\begin{eqnarray*}
 | |q ^x|^\alpha  -|q^y|^\alpha |\,|X|&\leq&   c \delta^{-1- \alpha} |\bar x-\bar y |^{-1+ {\gamma\over 2}}\ {\rm if } \ \alpha \leq 0, \ {\rm or}\  \alpha >1,\\
| |q ^x|^\alpha  -|q^y|^\alpha| \, |X|&\leq&   c\delta^{-1} |\bar x-\bar y |^{-1+ {\gamma\alpha \over 2}}\ { \rm if  } \ \alpha \in ]0,1[
\end{eqnarray*}
%
%

We have obtained, with the above choice of $\gamma$ that , as soon as $\delta$ is small enough,
                $|F(q^x, X)-F(q^y, X) | $ is small with respect to  
$c\delta^{-1- \alpha} | \bar x-\bar y |^{\tau-1} .$
We now treat the terms involving $b$ with similar considerations.  
If $\beta \geq 1$,          
         $$|b(\bar x)| ||q^x|^{\beta } - |q^y |^{\beta } | \leq | b|_\infty  |q^x-q^y| M^{\beta -1} \leq c\delta^{-\beta+1}| \bar x-\bar y |^{ \gamma \over 2},$$ 
         while if  $ \beta \leq 1$, 
$$ |b(\bar x)| ||q^x|^{\beta } - |q^y |^{\beta } | \leq | b|_\infty  c | \bar x-\bar y |^{ \gamma \beta \over 2}.$$
         
          
Observe also that
$$|b(\bar x)-b(\bar y) ||q^x|^\beta \leq c\ {\rm lip }\  b  |\bar x-\bar y| \delta^{-\beta}.$$
Finally, 
\begin{eqnarray*}   
|b(\bar x)-b(\bar y) ||q^x|^\beta &\leq& c|\bar x-\bar y|^{\tau-1}\delta^{-\beta+2-\tau}\\
&=& c|\bar x-\bar y|^{\tau-1}\delta^{-1-\alpha}\delta^{-\beta+3-\tau+\alpha}
\end{eqnarray*}         
Since  $3+ \alpha-\beta -\tau>0$ , this term is  also small wrt to  $\delta^{-1- \alpha} |\bar x-\bar y |^{\tau-1}$.

Putting all the estimates together we get
               \begin{eqnarray*}
              - g(\bar x)&\leq&  F(q^x, X) -b(\bar x) |q^x |^\beta \\
               &\leq &  F(q^y, -Y) - b( \bar y) |q^y |^\beta - c\delta^{-1-\alpha} |\bar x-\bar y |^{-1+ \tau} \\
               &\leq & -f( \bar y)-c\delta^{-1-\alpha} |\bar x-\bar y |^{-1+ \tau}.
               \end{eqnarray*}
This is clearly a contradiction as soon as $\delta$ is small enough since $f$ and $g$ are bounded. 
 \end{proof}

\section{Existence and uniqueness results for homogenous Dirichlet conditions.}\label{Scomp}
   
 The existence of solutions for (\ref{eq1}) with the boundary condition $\varphi = 0$ will be 
 classically obtained as a  consequence of  the existence of sub- and supersolutions,  
of some comparison result, and the Perron's  method. \\
We start with  a result on the existence of  sub and supersolutions of equations involving the Pucci's operators
$$
\begin{array}{c}
\mathcal{M}^+ (M)= A \sum_{\lambda_i>0} \lambda_i +a \sum_{\lambda_i<0} \lambda_i \\[2ex]
\mathcal{M}^- (M)= a \sum_{\lambda_i>0} \lambda_i +A \sum_{\lambda_i<0} \lambda_i 
\end{array}$$
defined for any $M\in \mathcal{S}_N$, $M\sim {\rm diag}(\lambda_1,\ldots ,\lambda_N)$.
 \begin{prop}\label{sursolution}
         Let $\lambda >0$ and $b>0$  be given.  For all $M>0$,   there exists a continuous function $\varphi \geq 0$ supersolution of 
\begin{equation}\label{soprasol}
         \left\{ \begin{array}{lc}
          - |\nabla\varphi|^\alpha \mathcal{M}^+(D^2 \varphi)   -b  |\nabla \varphi|^{\beta}  + \lambda \varphi^{1+ \alpha} \geq M & {\rm in } \ \Omega \\
           \varphi = 0 & {\rm on } \ \partial \Omega
           \end{array} \right.
           \end{equation}
         and symmetrically,  
         for all $M>0$,  there exists a continuous function $\varphi \leq 0$ subsolution
$$\left\{ \begin{array}{lc}
         -|\nabla\varphi|^\alpha {\cal M}^-(D^2 \varphi) +b  |\nabla \varphi|^{\beta} +  \lambda |\varphi|^{ \alpha} \varphi    \leq - M & {\rm in } \ \Omega \\
           \varphi = 0 & {\rm on } \ \partial \Omega.
           \end{array} \right.$$
 \end{prop}
 
 \begin{proof}[Proof of Proposition \ref{soprasol}] We shall construct explicitly a positive supersolution, the subsolution is just the negative of it. 
 
Since $\Omega$ is a smooth domain,   $d(x)$, the function distance from the boundary, is ${ \cal C}^2$ in  the 
neighborhood $\{ d (x)<\delta_0\}$ for some $\delta_0>0$, hence we will suppose to extend it to a ${ \cal C}^2$ function in $\Omega$, which is greater than $\delta_0$ outside of a $\delta_o$ neighborhood of the boundary.

Let us fix a positive constant $\kappa$ satisfying $\lambda \log (1 + \kappa)^{1+ \alpha} >  M$ and, for $C>\frac{2\kappa}{\delta_0}$, consider the function $\varphi (x)=  \log (1+ C d(x))$.
Observe that $|\nabla d | = 1$ in $\{d(x)<\delta_0\}$.

Suppose that   $\varphi$ is showed to be a supersolution in the set $\{ Cd(x)< 2 \kappa\}$. Then, the function
                $$ \phi(x) := \left\{ \begin{array}{cc}
                \log (1+ Cd(x) ) \ & {\rm in} \ \left\{C d (x)< \kappa \right\}\\
                 \log (1 + \kappa)& \ {\rm in}    \left\{C d (x)\geq  \kappa \right\}
                \end{array} \right.$$
is the required supersolution, being the minimum of two supersolutions.

We now prove that in $\{ Cd(x)< 2 \kappa\}$,  $\varphi$ is a supersolution. Easily, one gets
   $$ \nabla \varphi = { C\nabla d\over (1+ Cd)}, \ D^2 \varphi = {C D^2 d \over 1+ Cd} - {C^2 \nabla d\otimes \nabla d \over (1+ Cd)^2}.$$
 We suppose first that $\beta<\alpha+2$.
Let  $C_1$ be  such that $D^2 d \leq C_1{\rm Id} $. For $C$ satisfying furthermore
$$ {a C\over  2(1+ 2\kappa) } \geq A N C_1, \ \ 
                                {a \over 2} \left({C\over 1+ 2 \kappa}\right) ^{2+ \alpha-\beta } > 2 b , \ \ 
                  {aC^{2+ \alpha}  \over  4(1+ 2 \kappa)^{2+ \alpha} }  > M\, ,$$
one gets that 
                   $$ |\nabla \varphi|^\alpha  {\cal M}^+ (D^2 \varphi) + b  |\nabla  \varphi |^{\beta}  \leq -M\, ,$$
from which the conclusion follows.
                  
We now suppose that  $\beta = \alpha+2$. We begin to observe that the calculations above can be extended to $\beta = \alpha+2$ as soon as $b < { a \over 4}$. 
 So suppose that $\epsilon =  { a \over 4b}$, that $\varphi$ is some barrier for  
 the equation 
 $$- | \nabla \varphi |^\alpha { \cal M}^+( D^2 \varphi) + { a \over 4} | \nabla \varphi |^\beta \geq M\epsilon ^{1+ \alpha}$$
  Then 
 $\psi =  { \varphi \over \epsilon}$ satisfies 
 $$- | \nabla \psi |^\alpha { \cal M}^+ ( D^2 \psi ) + b  | \nabla \psi |^\beta \geq M . $$
 
  \end{proof}
Recall that, for $\alpha<0$,  the equation we are considering are singular. Hence, we need to treat differently the
solutions when the test function has a vanishing gradient. That is the object of the following lemma,
which  is proved on the model of \cite{BD1} where we treat the case $\beta = \alpha+1$.  
We give the detail of the proof for the convenience of the reader.
  \begin{lemme}\label{lemalphaneg}
Let $\gamma$ and  $b$ be    continuous functions. 
Assume that $v$ is a supersolution of 
 $$-|\nabla v|^\alpha \mathcal{M}^-( D^2 v) +b(x)  |\nabla v |^\beta+ \gamma (v)  \geq f\quad \mbox{in}\quad \Omega$$
such that, for some $C>0$ and $q\geq {\alpha+2 \over \alpha+1}$, $\bar x \in \Omega$ is  a strict local minimum
 of
 $v(x) + C |x-\bar x|^q$. 
Then 
$$f(\bar x) \leq \gamma (v(\bar x)).$$
 \end{lemme}
\begin{proof} Without loss of generality we can suppose that $\bar x=0$.\\
If $v$ is locally constant around $0$ the conclusion is the definition of  viscosity supersolutions. 
If $v$ is not locally constant,  since $q>1$,
for any $\delta>0$ sufficiently small, 
there exist $(z_\delta, t_\delta) \in B_\delta^2$  such that 
\begin{equation}\label{constant} v(t_\delta) > v(z_\delta) + C |z_\delta-t_\delta|^q.\end{equation}
The idea of the proof is to construct a test function near $0$ whose gradient is not zero.

Since $0$ is a strict minimum point of $v(x) + C |x|^q$, there exist $R>0$ and, for any $0<\eta <R$,  $\epsilon(\eta) >0$ satisfying 
       $$ \min_{  \eta\leq |x| \leq R } ( v(x) + C |x|^q ) \geq v(0) + \epsilon (\eta).$$
Let $\eta>0$ be fixed.  We choose $\delta=\delta(\eta)>0$ such that $C\delta^q\leq \frac{\epsilon (\eta)}{4}$. Note that $\delta\to 0$ as $\eta\to 0$.
With this choice of $\delta$, we have
             $$\min_{ |x| \leq R}  (v(x) + C |x-t_\delta|^q ) \leq v(0) + C |t_\delta|^q \leq v(0) + {\epsilon(\eta)  \over 4}.$$
On the other hand,  restricting further $\delta$ such that $q\delta C (R+1)^{ q-1} < {\epsilon(\eta)  \over 4}$, we get
$$\min_{\eta \leq  |x| \leq R} (v(x) + C |x-t_\delta|^q ) \geq \min_{\eta \leq  |x| \leq R} (v(x) + C |x|^q) - qC |t_\delta|   (R+  |t_\delta|)^{q-1} \geq v(0) + {3\epsilon(\eta) \over 4}.$$
This implies that   $\min_{ |x| \leq R}  (v(x) + C |x-t_\delta|^q ) $ is achieved in $B_\eta$. Furthermore,  it 
cannot be achieved in  $t_\delta$ by \eqref{constant}. Hence, there exists $y_\delta\in B_\eta$, $y_\delta \neq t_\delta$, such that 
               $$ v(y_\delta) + C |y_\delta-t_\delta|^q = \min_{|x|\leq R}( v(x) + C |x-t_\delta|^q).$$
 Let us now consider the test function
$$\varphi(z) = v( y_\delta) + C |y_\delta-t_\delta|^q -C |z-t_\delta|^q \, ,$$              
that touches $v$ from below at $y_\delta$. Since $v$ is a supersolution,  we obtain 
\begin{equation}\label{super}
 NA (q-1)  q^{\alpha+2}  C^{\alpha+1} |y_\delta-t_\delta|^{q(\alpha + 1)-(\alpha +2)} + 
                      C^\beta |b(y_\delta)|   |y_\delta-t_\delta|^{(q-1)\beta} +\gamma (v(y_\delta)) \geq f(y_\delta).
\end{equation}         
On the other hand, we observe that 
$$ v(y_\delta)\leq v(y_\delta) + C |y_\delta-t_\delta|^q \leq v(0) + C |t_\delta|^q\leq v(0) + C \delta^q.$$
By the lower semicontinuity of $v$, this implies that
$$v(y_\delta)\to v(0)\quad \hbox{as } \eta\to 0\, .$$
Thus, letting $\eta\to 0$ in \eqref{super}, by the continuity of $\gamma$ and $f$ it follows that
$$
\gamma(v(0))\geq f(0)\, .
$$
  \end{proof} 
We are now in a position to prove the following comparison principle that will be essential
 to the proof of the existence of the solution.

 \begin{theo}\label{comp} Suppose that $\Omega$ is a bounded  domain and  that $\gamma$ is a non decreasing function. Assume  that,  in $\Omega$, $u$ is an upper semicontinuous  bounded from above   viscosity subsolution of 
      $$-F(\grad u,D^2 u)  + b(x)|\nabla u|^\beta + \gamma (u) \leq  g$$
and  that $v$ is a  lower semicontinuous  bounded from below viscosity supersolution of 
       $$- F(\nabla v, D^2 v)  + b(x)|\nabla v|^\beta+ \gamma (v) \geq  f\, ,$$
with $f$ and $g$  bounded and $b$ Lipschitz continuous.\\
 Suppose furthermore that \\
-either $g\leq f$ and $\gamma$ is increasing ,\\
-or $g < f$. \\
 Then 
 $$ u\leq v\  \mbox{ on}\ \partial\Omega \ \Longrightarrow
  \ u\leq v\ \mbox{ in }\ \Omega.$$
          \end{theo}
\begin{proof} 
{\bf The case $\alpha \geq 0$.}
We use classically the doubling of variables.
So we define  for all $j \in N$, $\psi_j(x, y) = u(x)-v(y) -{j\over 2} |x-y|^2 $. Suppose by contradiction that $ u > v$ somewhere,  then the supremum of $u-v$ is strictly positive and achieved inside $\Omega$.  

Then one also has  $\sup \psi_j >0$  for $j$ large enough and it is achieved  on  $(x_j, y_j) \in \Omega^2$. 
Using Ishii's lemma , \cite{I1},  there exist $X_j$ and $Y_j$ in $S$ such that 
             $(j(x_j-y_j), X_j) \in \overline{J}^{2,+} u(x_j)$, 
              $(j(x_j-y_j), -Y_j) \in \overline{J}^{2,-} v(y_j)$. 
It is clear that Theorem \ref{theolip}  in section 2 can be extended to the case where $\Omega$ replaces $B(0, 1)$ and $\Omega^\prime \subset \subset \Omega$ replaces $B(0, r)$. Since $(x_j, y_j)$ converges to $( \bar x, \bar x)$, both of them  belong, for $j$ large enough, to some $\Omega^\prime $. 
We use Theorem  \ref{theolip}   to obtain that 
               $ j |x_j-y_j|$ is bounded. 
               
Indeed $u(x_j)-v(y_j)-{j\over 2} |x_j-y_j|^2 \geq \sup (u-v)$, hence
$$ {j\over 2} |x_j-y_j|^2 \leq u(x_j)-v(y_j) -\sup (u-v) \leq \sup (u-v)+ c |x_j-y_j| -\sup (u-v).$$
We obtain 
             \begin{eqnarray*}
              g(x_j)- \gamma ( u(x_j))  &\geq & - F(j(x_j-y_j),X_j) +  b(x_j) | j(x_j-y_j)|^\beta\\
              &\geq & -F(j (x_j-y_j), -Y_j)  + b(y_j)   | j(x_j-y_j)|^\beta + \omega (j |x_j-y_j|^2) \\
              &- &o(x_j-y_j) (j |x_j-y_j|)^\beta  \\
              &\geq &f(y_j) + o(j |x_j-y_j|^2)|(j|x_j-y_j|)^\alpha - \gamma (v(y_j)).\\
             \end{eqnarray*}
By passing to the limit,  one gets  on the point $\bar x$ limit of a subsequence of $x_j$ 
$$ g(\bar x) -\gamma ( u(\bar x))\geq f(\bar x)-\gamma(v(\bar x))$$
and in both cases we obtain a contradiction. 

{\bf The case $\alpha <0$.} 
We recall that in the case $\alpha <0$ one must use  a   convenient definition of viscosity solutions, see \cite{BD1}.

We suppose by contradiction that 
         $ \sup (u-v) >0$ then it is achieved inside $\Omega$,  and taking $q > {\alpha+2 \over \alpha +1}$
the function
$$\psi_j(x,y)=u(x)-v(y)-{j\over q}|x-y|^q$$
 has  also  a local maximum on $(x_j, y_j)$, with $x_j\neq y_j$.
Then, there are
$X_j$,
$Y_j$
$\in {\mathcal S}^N$ such that 
$$( j(|x_j-y_j|^{q-2} (x_j-y_j), X_j)\in
J^{2,+} u(x_j)$$

$$( j(|x_j-y_j|^{q-2} (x_j-y_j), -Y_j)\in J^{2,-}
v(y_j)$$

and $$ -4jk_j\left(\begin{array}{cc} I&0\\
0&I\end{array}\right)\leq \left(\begin{array}{cc} X_j&0\\
0&Y_j\end{array}\right)\leq 3jk_j\left( \begin{array}{cc}
I&-I\\ -I&I
\end{array}
\right)$$ where $$k_j =2^{q-3} q(q-1) |x_j-y_j|^{q-2}
.$$

 Note that :

(i) from the boundedness of $u$ and $v$ one deduces that $|x_j-y_j |\rightarrow 0$ as $j\rightarrow \infty$. Thus up to subsequence $(x_j, y_j)\rightarrow (\bar x, \bar x) $. 

 (ii) One has $\lim\inf \psi_j(x_j,y_j)\geq \sup  (u-v)$;
 
  (iii) $\lim\sup \psi_j(x_j,y_j)\leq\lim\sup u(x_j)-v(y_j)= u(\bar x)-v(\bar x) $
  
 (iv) Thus $j|x_j-y_j|^q\rightarrow 0$ as $j\rightarrow +\infty$ and $\bar x$ is a maximum point for $u-v$.
 
 Furthermore by the Lipschitz estimates in Theorem \ref{theolip}
 $ j |x_j-y_j|^{q}\leq u(x_j)-v(y_j) -\sup (u-v)  \leq c |x_j-y_j|$  since $(x_j, y_j)$ belong to a compact set inside $\Omega$. This implies that $j|x_j-y_j|^{q-1}$ is bounded.  
\bigskip

{\bf Claim: } {\it For $j$ large enough,
 there exist $x_j$ and $y_j$ such that $(x_j, y_j)$ is a maximum pair
for $\psi_j$ and 
$x_j\neq y_j$.}

 \noindent Indeed suppose that $x_j = y_j$. Then one would have
\begin{eqnarray*}
\psi_j (x_j, x_j)&=& u(x_j)-v(x_j)
\geq u(x_j)-v(y)-{j\over q} |x_j-y|^q;\\
\psi_j (x_j, x_j)&=& u(x_j)-v(x_j)
\geq u(x)-v(x_j)-{j\over q} |x-x_j|^q;\\
\end{eqnarray*}
 and then $x_j$ would be a local minimum for
$\Phi:=v(y)+{j\over q} |x_j-y|^q, $ 
{ \rm and\  similarly\   a\  local \ maximum\  for}\ 
$\Psi:=u(x)-{j\over q} |x_j-x|^q.$

We first exclude that $x_j$ is  both a strict local maximum and a strict local minimum.
Indeed  in that case, by Lemma \ref{lemalphaneg}
$$\gamma (v(x_j))\geq f(x_j), \ { \rm and} \ \gamma (u (x_j))\leq g( x_j).$$

This is a contradiction with the assumptions  on $\gamma$ and $f$ and $g$, once we pass to the limit when $j$ goes to $\infty$  since we get 
$$\gamma (v(\bar x)) \geq f( \bar x) \geq g(\bar x) \geq \gamma ( u( \bar x)).$$ 
Hence  $x_j$ cannot be   both a strict minimum for $\Phi$  and a 
strict maximum for $\Psi$.
Suppose that this is the case for $\Phi$, then  there exist
$\delta>0$ and $R> \delta$ such that $B(x_j, R)\subset \Omega$ and 

$$v(x_j) = \inf_{\delta\leq |x-x_j|\leq R} \{ v(x)+ {j\over q}
|x-x_j|^q\}.$$

Then if $y_j$ is a point on which the minimum above is achieved, one has 
$$v(x_j)=v(y_j)+{j\over q} |x_j-y_j|^q,$$ 
and  $(x_j, y_j)$ is still a maximum point for $\psi_j$. 

\bigskip We can now conclude. By  Ishii's  Lemma    there exist
$X_j$ and $Y_j$ such that 

$$\left(j|x_j-y_j|^{q-2} (x_j-y_j), {X_j}\right)\in J^{2,+} u(x_j)$$
and 

$$\left(j|x_j-y_j|^{q-2} (x_j-y_j), {-Y_j}\right)\in J^{2,-} v(y_j).$$

We can use the fact that $u$ and $v$ are respectively sub and super solution to obtain, (recalling that $j |x_j-y_j|^{q-1}$ is bounded) :

\begin{eqnarray*}
g(x_j) & \geq &F( j |x_j-y_j|^{ q-2} (x_j-y_j), X_j)+b(x_j) |j(x_j-y_j|^{q-1}|^{ \beta} +\gamma (u(x_j)) \\
&\geq & F( j |x_j-y_j|^{ q-2} (x_j-y_j), -Y_j)+b(y_j) |j(x_j-y_j|^{q-1}|^{ \beta}  \\
&- &{\rm lip} b |x_j-y_j|  |j|x_j-y_j|^{q-1} |^\beta   +\gamma (v(y_j)) + \left(- \gamma (v(y_j))+\gamma (u(x_j)) \right)\\
&\geq & f(y_j) -  \gamma (v(y_j))+\gamma (u(x_j)) + O(x_j-y_j) 
 \end{eqnarray*}
  Passing to the limit one obtains 
  $$ g(\bar x) \geq f(\bar x) - \gamma (v(\bar x))+\gamma (u(\bar x))$$
  We have obtained a contradiction in each of the cases "$f>g$ and $\gamma$ is non decreasing", or "$f \geq g$ and $\gamma$ is increasing". 
\end{proof}   
We derive from the construction of barriers and the comparison theorem the following existence result for Dirichlet homogeneous boundary condition 

 \begin{theo}
  Let $\Omega$ be a bounded ${ \cal C}^2$ domain, 
   $\lambda >0$, $f \in { \cal C}( \overline{ \Omega})$ and $b \in W^{1, \infty} ( \Omega)$. Under the assumptions (H1) and (H2),  there exists a unique $u $ which satisfies 
    \begin{equation}\label{eqqq}
\left\{\begin{array}{lc}
      -F(  \nabla u, D^2 u)+b(x)  |\grad u|^\beta  + \lambda |u |^\alpha u =   f& \hbox{ in} \  \Omega \\
      u= 0&\ \hbox{ on} \  \  \partial \Omega 
       \end{array}\right.
\end{equation}
 Furthermore, $u$ is Lipschitz continuous up to the boundary. 
 \end{theo}
  \begin{proof} 
   The existence result is an easy consequence of the comparison principle and Perron's method adapted to our  framework, as it is done in \cite{BD1}. 
    In order to prove that $u$ is Lipschitz continuous up to the boundary,  observe that, by construction,  there exist $C>c>0$ such that
   $c d \leq u \leq Cd$, where $d$ is the distance function from $\partial \Omega$. Then, consider the function 
      $$\phi(x, y) = u(x)-u(y)-M \omega ( |x-y|)\, , $$
where  $\omega$ has been defined in (\ref{omega}), and suppose that $M >2 C$. Arguing as in the proof of Theorem \ref{theolip},  we assume by contradiction that  $\phi$ is positive somewhere, say on $( \bar x, \bar y)$. Then, neither $\bar x$ nor $\bar y$ belong to the boundary, due to the inequality, for $y \in \partial \Omega$,
        $$u(x) \leq Cd(x)\leq { M \over 2} d(x) \leq M \omega ( |x-y|).$$ 
A similar reasoning proves that $ \bar x$ cannot belong to the boundary. 
The rest of the proof runs  as in Theorem (\ref{theolip}), with  even simpler computations, since we do not need the additional term  $ |x-x_o|^2+|y-y_o|^2$  in the auxiliary function $\phi$.
 
\end{proof}
    
We end this section with some Strong maximum principle and Hopf principle. 
 
\begin{theo}\label{strong} Suppose that $u$ is a non negative solution in $\Omega$ of 
  $$-|\grad u|^\alpha \mathcal{M}^-(D^2 u)  + b(x)|\nabla u|^\beta \geq 0$$
   Then either $u>0$ inside $\Omega$, or $u \equiv 0$. 
Moreover if   $\bar x$ is in $\partial \Omega$ so that an interior sphere condition holds and $u(\bar x)=0$, then 
    $" \partial_n u(\bar x)"  <0$
\end{theo}
\begin{proof}
Suppose that there exists an interior point $x_1$ such that $u(x_1)=0$. 
We can choose $x_1$ such that there exists  $x_o$  satisfying
 $|x_o-x_1| = R$, the ball $B(x_o, 2R) \subset \Omega$ and $u>0$ in $B(x_o, R)$. 
 Since $u$ is lower semi-continuous , let 
  $\delta < \inf (1, \inf_{ B(x_o, {R \over 2})} u)$, and define  in the crown $B(x_o, 2R) \setminus B(x_o, {R\over 2})$
  
  $ v(r) = \delta (e^{-c r}-e^{-c R})$ where 
  $c > {2 (N-1) A \over Ra}$, and $ { a\over 2} c^{2+ \alpha-\beta }> |b |_\infty$, if $\beta < \alpha +2$, ( which implies that  $\delta^{1+ \alpha} { a\over 2} c^{2+ \alpha-\beta }> \delta^\beta |b |_\infty$), if $\beta = \alpha +2$, note that we can take $\delta$ so that 
  $\delta ^{ 1+ \alpha-\beta } {a \over 2} > |b|_\infty$ . Then  one has 
  $v \leq u$ on the boundary of the crown, and 
  $$-|  v^\prime |^\alpha ( a v^{ \prime \prime }+ A{N-1 \over r} v^\prime ) + |b|_\infty |v^\prime |^\beta <0$$ 
  
Using the comparison principle one gets that 
$u \geq v$. Then $v$ is a ${ \cal C}^2$ function which achieves $u$ on below on $x_1$ and this is  a contradiction with the fact 
that $u$ is a super-solution. 
The last statement is proved .
Suppose that $\bar x$  is on the boundary and consider an interior sphere $B(x_o, R) \subset \Omega$ with $|x_o - x_1| = R$, and the function $v$ as above. We still have by construction that $v \geq u$. Then taking for$ h > 0$ small, $x_h = hx_o +(1-h)x_1$ one has $|x_h -xo| = (1 - h)R$ and
$${u(x_h) - u(x_1)\over x_h-x_1}  \geq { v(x_h) - v(x_1)\over x_h-x_1}  \geq cRe^{-cR} >0$$ 
which implies the desired Hopf's principle.
\end{proof}

\section{Non homogeneous boundary conditions} \label{Nonh}
 In order to  obtain solutions for the non homogeneous boundary condition, we need  the construction of barriers, and a Lipschitz estimate near the boundary, as follows 
          \begin{lemme}\label{lemboundary} Let $B^\prime$ be the unit ball in $\R^{N-1}$, and let $\varphi\in W^{1, \infty} ( B^\prime)$.
Let  $\eta\in{\cal C}^2(B^\prime)$ such that $\eta(0)=0$ and $\nabla \eta (0)=0$. Let $d$ be the distance to the 
hypersurface $\{ x_N = \eta(x^\prime)\}$. 

Then,  for all $r<1$ and for all $\gamma <1$, there exists   
$\delta_o$ depending on $\|f\|_\infty$, $a$, $A$, $\|b\|_\infty$, $\Omega$, $r$ and  $\rm{Lip}\varphi$,  $|u|_\infty$    such that for all $\delta <\delta_o$   ,
if $u$  be a USC bounded by above sub- solution   of 
\begin{equation}\label{eqqqfi}
\left\{\begin{array}{lc}
      -F(  \nabla u, D^2 u)+b |\grad u|^\beta  \leq  f& {\rm in} \  B\cap \{ x_N >\eta(x^\prime)\}\\
         u\leq \varphi &\ {\rm on} \  \  B\cap  \{ x_N = \eta(x^\prime)\}
       \end{array}\right.
\end{equation}
such that $u  \leq 1$ then it satisfies 
$$ u(x^\prime, x_N)\leq \varphi(x^\prime)+    \log ( 1+    {2 \over \delta} d)\ \mbox{ in}  
\ B_r(0)\cap\{x_N =\eta(x^\prime)\} .$$  
 We have a symmetric result for supersolutions bounded by below. 
 \end{lemme}
 \begin{proof} 
{\bf First case.} We suppose that $\beta < \alpha+2$. 
We  also write the details of the proof for $\varphi=0$.  
The changes to bring in the case $\varphi\neq 0$ 
will be given at the end of the proof, the detailed calculation being left to the reader. 

It is  sufficient to consider the set where $ d(x)  < \delta$ since the assumption  
$\|u\|_\infty\leq 1$ implies the result elsewhere. 

We begin by choosing  $\delta< \delta_1$,  such that on $d(x) < \delta_1$  the distance 
is ${\cal C}^2$ and  satisfies 
$|D^2 d|\leq C_1$.  We shall also later choose  $\delta$ smaller depending  of 
$(a, A, \|f\|_\infty, \|b\|_\infty  ,  N)$.

In order to use the comparison principle we want to construct $w$ 
a super solution of 
\begin{equation}\label{ww}
-| \nabla w|^\alpha  {\cal M}^+ (D^2 w)-b|\nabla w|^\beta  \geq \|f\|_\infty,\   {\rm in} \  B \cap \{ x_N > \eta(x^\prime),\ d(x)< \delta\}
\end{equation}
such that $w\geq u$ on $\partial (B \cap \{ x_N > \eta(x^\prime),\ d(x)< \delta\})$.

We then suppose that $\delta < { 1-r\over 9}$, define $C = {2\over \delta}$  and 

$$w(x) = \left\{ \begin{array}{lc}
           \log ( 1+ C d)&  {\rm for} \  |y  |  < r\\
              \log ( 1+ C d)+ {1\over  (1-r)^3} (|x|- r)^3&\ {\rm for} \  1\geq  | x | \geq r.
             \end{array}\right.$$
In order to prove the boundary condition, let us observe that, 

 on $\{ d(x)=\delta\} $,
$w\geq  \log 3  \geq 1\geq u$, 

on  $\{|x | = 1\}\cap\{d(x) < \delta\}$, $w \geq {1\over (1-r)^3}(1-r)^3 \geq u $ and finally

on $B\cap \{ x_N = \eta (x^\prime)\}$, $w\geq 0 = u$.

\noindent We need to check that $w$ is a super solution.
For that aim,  we compute 
  $$\nabla w =\left\{ \begin{array}{lc}
  {C \over 1+ Cd} \nabla d & {\rm  when} \ |x|< r\\
   {C \over 1+ Cd} \nabla d +   {x \over |x| }{ 3\over (1-r)^3}(|x|-r)^2 & {\rm  if} \ |x| > r.
  \end{array}\right. $$

Note that $|\nabla  w|\geq{ C \over 2( 1+ C d) }\geq    {1\over 3\delta}$ since $\delta\leq\frac{1-r}{9}$.  One also has 
$ | \nabla w | \leq {3C \over 2(1+ C d)} $. 
By construction $w$ is ${\cal C}^2$ and
$$D^2 w ={ CD^2 d \over 1+ Cd} -{C^2 \nabla d \otimes \nabla d \over (1+Cd)^2}+ 
  H(x)$$
where $|H(x)|\leq {6\over (1-r)^2}+\frac{3N}{r(1-r)}$. 
In particular   
\begin{eqnarray*}
-{\cal M}^+ ( D^2 w)& \geq & a {C^2 \over (1+ C d)^2}-A{  C |D^2d |_\infty\over 1+ Cd}-  A\left({6\over (1-r)^2}+\frac{3N}{r(1-r)} \right) \\
&\geq &  a {C^2 \over (1+ Cd)^2}-A C |D^2d |_\infty-A\left({6\over (1-r)^2}+\frac{3N}{r(1-r)} \right)\\
&\geq & a {C^2 \over 4(1+ Cd)^2}
\end{eqnarray*}
 as soon as $\delta$  is small enough, depending only on $r$, $A$,$ a$. 
 
   Hence  ( we do the computations for $\alpha >0$ and leave to the reader the case $\alpha <0$, which can easily be deduced easily since 
   ${ C \over 2( 1+ C d) }\leq | \nabla w | \leq {3C \over 2(1+ C d)} $). 
    \begin{eqnarray*}
-| \nabla w |^\alpha { \cal M}^+ ( D^2 w) -b | \nabla w |^\beta &\geq &  a {C^{2+ \alpha} \over 2^{2\alpha+2} ( 1+ Cd)^{2+ \alpha} }-b \left( { C\over 2(1+Cd)}\right)^\beta \\
 &\geq &a {C^{2+ \alpha} \over 2^{4+  \alpha} ( 1+ Cd)^{2+ \alpha} }
 \geq |f|_\infty
 \end{eqnarray*}
 as soon as $\delta$ is small enough in order that 
 \begin{equation}\label{eqbsmall}b < a 2^{ \beta-2\alpha-4} \left({C \over 3}\right)^{ \alpha+2-\beta}
 \end{equation} 
  and 
  so that 
  $a {C^{2+ \alpha} \over 2^{4+  \alpha} ( 1+ Cd)^{2+ \alpha} }> |f|_\infty$.

By the comparison principle, Theorem \ref{comp}, $u\leq w$ in $B \cap \{ x_N > \eta ( x^\prime)\}\cap \{d(x)< \delta\}$. 

Furthermore  the  desired lower bound  on $u$   is easily deduced by considering 
$-w$ in place of $w$ in the previous computations and restricting to $B_r \cap \{ x_N > \eta ( x^\prime)\}$.
This  ends the case $\varphi\equiv 0$.

 To treat the case where $\varphi$ is  non zero, let $\psi$ be a solution of 
 $${ \cal M}^+ ( \psi) = 0, \ \psi = \varphi\ { \rm on} \ \partial \Omega$$
 It is known that $\psi\in { \cal C}^{2}( \Omega)$, $\psi$ is Lipschitz, $| \nabla \psi | \leq K | \nabla \varphi
|_\infty$. Then in the previous calculation it is sufficient to define 
 as soon as $\delta$ is small enough in order that 
 ${1 \over 3 \delta}
> 2 K | \nabla \varphi|_\infty$,
$$w(y) = \left\{ \begin{array}{lc}
           \log ( 1+ C d)+ \psi(x)&  {\rm for} \  |x  |  < r\\
              \log ( 1+ C d)+ {1\over  (1-r)^3} (|x|- r)^3+ \psi(x)&\ {\rm for} \   | x | \geq r.
             \end{array}\right.$$
                           And we choose $C= {2 \over  \delta} $ large enough in order that 
               ${ C \over  2} > 2(| \nabla \psi |_\infty + {1\over 1-r})$ and also large enough in order that 
                $-| \nabla w |^\alpha { \cal M}^+ ( D^2 w) -b | \nabla w |^\beta > |f|_\infty$ which can be easily done by the same argument as before,  using 
             \begin{eqnarray*}
             -| \nabla w |^\alpha { \cal M}^+ ( D^2 w) -b | \nabla w |^\beta  &\geq &-2^{-| \alpha|}  | \nabla (   \log ( 1+ C d)|^\alpha  { \cal M}^+  (   \log ( 1+ C d)) \\
             &-& b 2^{-\beta } |\nabla (   \log ( 1+ C d)|^\beta. 
             \end{eqnarray*}

\noindent{\bf Second case} Suppose now that $\beta  = \alpha +2$. 
 It is sufficient to construct a convenient super-solution  $w$. 
             
Recall that the previous proof used (\ref{eqbsmall}), which reduces when  $\alpha+2 = \beta$, to the condition 
$b < a 2^{ \beta-2\alpha-4}$. Let then $\epsilon = {a 2^{ \beta-2\alpha-4}  \over b}$, and let $w$ which  equals ${\varphi\over \epsilon} $ on  the boundary, and satisfies 
           $$ - |\nabla w |^\alpha { \cal M}^+ ( D^2 w) -\epsilon b | \nabla w |^\beta \geq \epsilon ^{1+ \alpha}|f|_\infty $$
            Then we obtain by the comparison principle that 
            ${ u\over \epsilon}  \leq w$ and then $u \leq \epsilon w$.   
           
           \end{proof}

 This enables us to prove   the following  Lipschitz estimate up to the boundary.
  \begin{prop}\label{lipbo} Let $\varphi$ be a Lipschitz continuous function on the part $T = B(0,1) \cap \{ x_N = \eta(x^\prime)\}$, that $f \in { \cal C} ( \overline{ \Omega})$ and $b$ is Lipschitz continuous.   
Suppose that $u$ and $v$ are respectively sub and supersolution  which  satisfy (\ref{eqqqfi}) in $B(0,1) \cap \{ x_N > \eta(x^\prime)\}$, with 
$u = \psi=  v$ on $T$. 

Then, for all $r<1$,   there exists
$c_r>0$   depending on $r,a, A, {\rm Lip}(b)$ and $N$  such that,  for all $x$ and $y$ in $B_r \cap \{ x_N \geq  \eta(x^\prime)\})$,
$$ u(x) -v(y) \leq  \sup_{B_1 \cap \{ y_N \geq  \eta(y^\prime)\})} (u-v) + c_r \left( |f|_\infty^{1\over 1+ \alpha} + |u|_\infty + |\psi|_{ W^{1, \infty} (T)}\right)|x-y| $$
In particular, when $u$ is a solution we have a Lipschitz local estimate up to the boundary. 
\end{prop}
  \begin{proof}
  We must prove    H\"older's case  and deduce from it the Lipschitz one as  in the proof of Theorem \ref{theolip} . We do not give the details, in  each of the H\"older's and Lipschitz case,  we define in $B(0, 1) \cap \{ x_N > \eta(x^\prime)\}$
                 $$\phi(x, y) = u(x) - v(y) -\sup_{B_1\cap  \{ x_N > \eta(x^\prime)\}} ( u-v) -M\omega (|x-y|) -L |x-x_o|^2-L |y-x_o|^2,$$
                  where $\omega$ is as in the proof of Theorem \ref{theolip}.
              We want to prove that $\phi \leq 0$, which will classically imply the result.   We argue by contradiction  and  need to prove first  that if $( \bar x, \bar y)$ is a maximum point for $\phi$, then  neither $\bar x$ nor $\bar y$ belongs to $x_N = \eta(x^\prime)$. We notice that,  if $\bar y \in \{ x_N = \eta(x^\prime)\}$, then
              \begin{eqnarray*}
              u(\bar x)-\psi ( \bar y) &\geq& \sup_{B_1\cap  \{ x_N > \eta(x^\prime)\}} ( u-v) + M\omega (|x-y|) \\
              &\geq& u( \bar y)-v( \bar y)+ M \omega ( | \bar x-\bar y|)
                       \end{eqnarray*}
 which contradicts  Lemma  \ref{lemboundary}  for $M$ large enough.  
The case when $\bar x \in \{ x_N= \eta(x^\prime)\}$ is excluded in the same manner, and the rest of the proof follows the lines of Theorem \ref{theolip}
  \end{proof} 
\begin{rema} {\rm  When the boundary condition is prescribed on the whole boundary,   we have a simpler proof, as we noticed in the homogeneous case, taking 
                $$\phi(x,y) = u(x) - v(y) -\sup_{B_1\cap  \{ x_N > \eta(x^\prime)\}} ( u-v) -M\omega (|x-y|)\, . $$
In this case the localizing terms are not needed, since  it is immediate to exclude that the 
maximum point $(\bar x , \bar y)$ is on the boundary. }
                 \end{rema}

\textbf{Acknowledgment}.
Part of this work  has been done while  the first and third authors were visiting the University of Cergy-Pontoise and 
the second one was visiting  Sapienza University of Rome, supported by INDAM-GNAMPA and Laboratoire AGM Research Center in Mathematics of the University of Cergy-Pontoise.

          \end{document}